\documentclass[a4paper,reqno]{article}
\usepackage{amsmath,amsfonts,amssymb,amsthm,bbm,graphicx,enumerate,geometry}
\usepackage{mathrsfs}
\usepackage{caption}
\usepackage{float}

\usepackage{xcolor}
\usepackage{listings}

\definecolor{bgcolor}{RGB}{250,250,250}
\definecolor{softgreen}{RGB}{0,128,0}
\definecolor{softblue}{RGB}{0,102,204}
\definecolor{softgray}{RGB}{150,150,150}
\definecolor{matstring}{RGB}{163, 21, 21}

\lstdefinelanguage{Matlab}{
  morekeywords={
    break,case,catch,continue,else,elseif,end,for,function,global,if,
    otherwise,persistent,return,switch,try,while,zeros,ones,plot,exp,cos,sin,
    sqrt,meshgrid,quiver,streamline,hold,axis,real,imag,abs
  },
  sensitive=true,
  comment=[l]{\%},
  morestring=[b]'
}

\lstdefinestyle{matlabIDE}{
  language=Matlab,
  backgroundcolor=\color{bgcolor},
  basicstyle=\ttfamily\small,
  frame=single,
  rulecolor=\color{softgray},
  framerule=0.6pt,
  keywordstyle=\color{softgreen}\bfseries,
  commentstyle=\color{softgray}\itshape,
  stringstyle=\color{matstring},
  numberstyle=\tiny\color{softgray},
  numbers=left,
  numbersep=8pt,
  breaklines=true,
  showstringspaces=false,
  showtabs=false,
  tabsize=4,
  captionpos=b,
  keepspaces=true,
  columns=flexible
}

\usepackage{subfigure} 

\usepackage{color}

\usepackage{framed} 
\usepackage{hyperref}  
\hypersetup{
    linktoc=page,
    linkcolor=red,          
    citecolor=blue,        
    filecolor=blue,      
    urlcolor=cyan,
   colorlinks= true          
}

\usepackage[
  backend=biber,
  style=alphabetic,
  autocite=inline,
]{biblatex}

\numberwithin{figure}{section}
\numberwithin{equation}{section}

\newtheorem{thm}{Theorem}[section]

\newtheorem{defn}[thm]{Definition}

\newtheorem{remark}[thm]{Remark}

\renewcommand{\epsilon}{\varepsilon}

\def\<#1{\langle #1\rangle}

\title{Irregular traces of multiple SLE(0) systems with multiple marked points}
\author{
    Jiaxin Zhang\footnotemark[1]
}

\begin{document}

\maketitle
\footnotetext[1]{\textbf{zhangjx@caltech.edu}, Department of Mathematics, California Institute of Technology}

\begin{abstract}

In this supplementary note, we study the traces of multiple SLE(0) systems with two or more additional marked points.

For general chordal configurations, the traces correspond to the real locus of real rational functions; in the radial case, they correspond to the horizontal trajectories of residue-free quadratic differentials. In both settings, we establish the regularity of the trajectories near singularities: no spiraling occurs, and no two trajectories asymptotically converge to the same direction.

Moreover, in the radial case with non-zero spin at the marked interior point, we show that the spin induces a spiraling behavior at the marked interior point.

However, this regularity breaks down when multiple interior marked points are present. In such cases, trajectories may asymptotically approach the same direction, and spiraling can occur even in the absence of spin. We present explicit counterexamples generated using MATLAB, with code provided for reference.
\vspace{1em}

\end{abstract}

\newpage
\tableofcontents

\newpage

\maketitle

\section{Introduction}
The Schramm–Loewner evolution SLE($\kappa$) for $\kappa > 0$ is a one-parameter family of random, conformally invariant curves in the plane that describe interfaces in critical models from statistical physics. This framework was introduced in \cites{Sch00, LSW04, Smi06, Sch07, SS09}. Conformal field theory (CFT), a quantum field theory invariant under conformal transformations, has also been extensively employed to analyze critical phenomena; see, for example, \cites{Car96, FK04}. 

SLE and multiple SLE systems can be coupled to conformal field theory through the SLE–CFT correspondence. This correspondence provides a powerful framework for predicting behaviors and computing quantitative observables of SLE($\kappa$) and multiple SLE($\kappa$) systems from a CFT perspective, as demonstrated in works such as \cites{BB03a, Car03, FW03, FK04, Dub15a, Pel19}. The parameter $\kappa$ measures the fractal roughness of the curves and determines the central charge of the associated CFT via the relation
\[
c(\kappa) = \frac{(3\kappa - 8)(6 - \kappa)}{2\kappa}.
\]

In recent years, there has been growing interest in multiple chordal SLE systems, as discussed in \cites{Dub06, KL07, Law09b, FK15a, PW19, PW20}. Multiple radial SLE systems have also been actively studied, with notable contributions including \cites{HL21, WW24, MZ24, MZ25, Zh25th}. Generalizations to arbitrary multiple chordal SLE configurations have appeared in \cites{Z25a, Z25b}.

In particular, multiple SLE(0) systems, which arise as deterministic limits of multiple SLE($\kappa$) as $\kappa \to 0$, have attracted attention. It has been shown that for multiple radial SLE(0), the traces are horizontal trajectories of residue-free quadratic differentials, while for general multiple chordal SLE(0) systems, the traces correspond to the real locus of rational functions.

Additionally, in the radial case, if a non-zero spin is assigned to an interior marked point, the induced spin causes the corresponding trajectory to spiral at the origin.

In this paper, we introduce a general framework for multiple SLE(0) systems with multiple marked points. More precisely, we define these systems with respect to a symmetric divisor, following constructions similar to those of multiple chordal and radial SLE(0) systems. 

Consider distinct real growth points \( \boldsymbol{x} = \{x_1, x_2, \ldots, x_n\} \subset \mathbb{R} \) and additional marked points \( \boldsymbol{q} = \{q_1, q_2, \ldots, q_m\} \subset \mathbb{C} \), which are closed under complex conjugation (i.e., \( q_j \in \boldsymbol{q} \) implies \( \overline{q_j} \in \boldsymbol{q} \)). We associate to these points a divisor
\[
\boldsymbol{\sigma} = \sum_{k=1}^n x_k + \sum_{j=1}^m \sigma_j \cdot q_j,
\]
where \( \sigma_j \in \mathbb{C} \). The divisor \( \boldsymbol{\sigma} \) is called \emph{symmetric} if it is invariant under complex conjugation:
\[
\boldsymbol{\sigma}^* := \sum_{k=1}^n x_k + \sum_{j=1}^m \overline{\sigma_j} \cdot \overline{q_j} = \boldsymbol{\sigma}.
\]

The behavior of the traces of such systems becomes highly intricate when the charges $\sigma_j$ are not half-integers. 

In the special case where all \( \sigma_j \in \tfrac{1}{2}\mathbb{Z} \), the SLE(0) traces can be described as horizontal trajectories of a suitable quadratic differential. However, even in this half-integer setting, the regularity of the trajectories may break down in the presence of multiple interior marked points. For instance, two trajectories may asymptotically converge in the same direction, or exhibit spiraling behavior—even in the absence of spin.

Explicit counterexamples illustrating these phenomena are provided in Section~\ref{examples}. These examples were generated using \textsc{MATLAB}, and the associated code is included for reference.

\section{Multiple SLE(0) systems associated with a symmetric divisor}

\subsection{Classical limit of Coulomb gas correlation}
This subsection is based on Subsection 2.5 of \cite{MZ24}. We construct the Coulomb gas correlation for $\kappa=0$, which will be used in the definition of multiple SLE(0) systems.

\begin{defn}[Normalized Coulomb gas correlations for a divisor on the Riemann sphere] Let the divisor
$$
\boldsymbol{\sigma}=\sum \sigma_j \cdot z_j,
$$
where $\left\{z_j\right\}_{j=1}^n$ is a finite set of distinct points on $\widehat{\mathbb{C}}$ and $\sigma_j \in \mathbb{C}, j=1,2,\ldots,n$.
The normalized Coulomb gas correlation $C[\boldsymbol{\sigma}]$ is a differential of conformal dimension $\lambda_j$ at $z_j$ by

Let $\lambda(\sigma)=\sigma^2+2\sigma  \quad(\sigma \in \mathbb{R})$.
\begin{equation}
\lambda_j=\lambda_b\left(\sigma_j\right) \equiv \sigma_j^2+2\sigma_j  ,
\end{equation}

whose value is given by

\begin{equation}
C[\boldsymbol{\sigma}] =\prod_{j<k}\left(z_j-z_k\right)^{2\sigma_j \sigma_k} ,
\end{equation}
where the product is taken over all finite $z_j$ and $z_k$. The Coulomb gas correlation is multivalued. To make it single-valued, one must choose a branch of the logarithm, i.e., fix a branch cut convention.

\end{defn}

\begin{remark}
    The normalized Coulomb gas correlation can be viewed as taking the $\kappa \rightarrow 0$ limit of the divisor $\sqrt{2\kappa}\boldsymbol{\sigma}$ satisfying the neutrality condition ($NC_b$), the Coulomb gas correlation function $C_{(b)}[\boldsymbol{\sigma}]^{\kappa}$, and conformal dimension $\kappa \lambda_j$.
\end{remark}

\begin{defn}[Neutrality condition when $\kappa=0$] A divisor $\boldsymbol{\sigma}: \widehat{\mathbb{C}} \rightarrow \mathbb{R}$ satisfies the neutrality condition if
\begin{equation}
\int \boldsymbol{\sigma}=-2.
\end{equation}
\end{defn}
\begin{thm}
Under the neutrality condition $\int\boldsymbol{\sigma}=-2$, the normalized Coulomb gas correlation differentials $C[\boldsymbol{\sigma}]$ are Möbius invariant on $\hat{\mathbb{C}}$.
\end{thm}

\begin{proof}
By direct computation, similar to the $\kappa>0$ case.

\begin{defn}[Symmetric divisor]
Let $\boldsymbol{\sigma} = \sum \sigma_j \cdot z_j$ be a divisor on the Riemann sphere $\widehat{\mathbb{C}}$. We say that $\boldsymbol{\sigma}$ is \emph{symmetric with respect to the real line} $\mathbb{R}$ if
\[
\boldsymbol{\sigma}^* = \boldsymbol{\sigma},
\]
where $\boldsymbol{\sigma}^* = \sum \overline{\sigma_j} \cdot \overline{z_j}$ is the reflection of $\boldsymbol{\sigma}$ under complex conjugation.
\end{defn}

\end{proof}

\subsection{Multiple SLE(0) with multiple marked points}
\begin{defn}[Multiple chordal SLE(0) Loewner chain] \label{multiple chordal SLE(0) Loewner chain via stationary relations}
In the upper half plane $\mathbb{H}$, given growth points $\boldsymbol{x}=\{x_1,x_2,\ldots,x_n\}$ on the real line $\mathbb{R}$,  $\boldsymbol{q}=\{q_1,q_2,\ldots,q_m\}$ closed under conjudgation. A symmetric divisor $\boldsymbol{\sigma}=\sum_{k=1}^{n} x_k+ \sum_{k=1}^{m} \sigma_k \cdot q_k$ where $\sigma_ k \in \mathbb{C}, k =1,2,\ldots,m$.

Let $\boldsymbol{\nu}=\left(\nu_1, \ldots, \nu_n\right)$ be a set of parametrizations of the capacity, where each $\nu_i:[0, \infty) \rightarrow[0, \infty)$ is assumed to be measurable. 

In the upper half plane $\Omega=\mathbb{H}$, we define the multiple SLE(0) 
Loewner chain as a normalized conformal map $g_t=g_t(z)$, with $g_{0}(z)=z$ whose evolution is described by the Loewner differential equation:
\begin{equation}
\partial_t g_t(z)=\sum_{j=1}^n \frac{2\nu_j(t)}{g_t(z)-x_j(t)}, \quad g_0(z)=z,
\end{equation}
and the driving functions $x_j(t), j=1, \ldots, n$, evolve as

\begin{equation}\label{multiple SLE 0 driving}
\dot{x}_j= \nu_j(t) \frac{\partial {\rm log} \mathcal{Z}(\boldsymbol{x},\boldsymbol{q})}{\partial x_j}+\sum_{k \neq j} \frac{2\nu_k(t)}{x_j-x_k}
\end{equation}
where
\begin{equation}
    \mathcal{Z}(\boldsymbol{x}, \boldsymbol{q})= C[\boldsymbol{\sigma}]= \prod_{1 \leq i<j \leq n}(x_i-x_j)^{2} \prod_{1 \leq i<j \leq m}(q_i-q_j)^{2\sigma_i\sigma_j} \prod_{i=1}^{n} \prod_{j=1}^m\left(x_i-q_j\right)^{2\sigma_j}
\end{equation}
The logarithmic derivative of $\mathcal{Z}(\boldsymbol{x}, \boldsymbol{q})$ with respect to $x_j$ (treating $x$ and $q$ as independent variables) is given by:
\begin{equation}
   \frac{\partial \mathcal{Z}(\boldsymbol{x},\boldsymbol{q})}{\partial x_j}=\sum_{k\neq j} \frac{2}{x_j-x_k}-2\sum_{l} \frac{\sigma_l}{x_j-q_l}
\end{equation}
for $j=1,\ldots,n$,
The flow map $g_t$ is well-defined up to the first time $\tau$ at which $x_j(t)=x_k(t)$ for some $1 \leq j<$ $k \leq n$. For each $z \in \mathbb{C}$, the process $t \mapsto g_t(z)$ is well-defined up to the time $\tau_z \wedge \tau$, where $\tau_z$ is the first time at which $g_t(z)=z_j(t)$. The hull associated with this Loewner chain is denoted by

$$
K_t=\left\{z \in \overline{\mathbb{H}}: \tau_z \leq t\right\}
$$

\end{defn}
\subsection{Traces as horizontal trajectories of quadratic differentials for half-integer charges distribution}

While the SLE(0) traces exist for arbitrary symmetric charge configurations $\boldsymbol{\sigma} = \sum_{k=1}^{n}x_k +\sum_{k=1}^{m}\sigma_k \cdot q_k$, they are not always characterized by rational functions or quadratic differentials.

We show that when all charges $\sigma_j$ are half-integers, the corresponding traces can be described as horizontal trajectories of a meromorphic quadratic differential.

This characterization holds in particular for general chordal SLE(0) systems with one additional boundary marked point, and for radial SLE(0) systems with a single interior marked point.

However, in the presence of two or more additional marked points, this regularity may break down: trajectories can exhibit irregular behavior, such as asymptotically converging to the same direction or developing spirals, even without introducing spin.

\begin{defn}[Quadratic differentials with prescribed zeros and poles]
\label{trace quadratic differential}
Let \( \boldsymbol{x} = \{x_1, x_2, \ldots, x_n \} \) be a set of distinct points on the real line \( \mathbb{R} \), and let \( \boldsymbol{q} = \{q_1, \ldots, q_m\} \subset \widehat{\mathbb{C}} \) be a set of distinct points closed under conjudgation. Consider a symmetric divisor
\[
\boldsymbol{\sigma} = \sum_{k=1}^{n} x_k + \sum_{j=1}^{m} \sigma_j \cdot q_j,
\]
where \( 2\sigma_j \in \mathbb{Z} \) for all \( j = 1, \ldots, m \). 

We define the class of quadratic differentials \( \mathcal{QD}(\boldsymbol{x}, \boldsymbol{q}) \) as those meromorphic quadratic differentials \( Q(z)dz^2 \) on the Riemann sphere \( \widehat{\mathbb{C}} \) satisfying the following properties:

\begin{enumerate}
    \item \( Q(z)dz^2 \) is symmetric under complex conjugation, i.e.,
    \[
    \overline{Q(\bar{z}) \, d\bar{z}^2} = Q(z) \, dz^2.
    \]

    \item \( Q(z) \) has simple zeros of order 2 at each \( x_k \in \boldsymbol{x} \), for \( k = 1, \ldots, n \).

    \item At each \( q_j \in \boldsymbol{q} \), \( Q(z) \) has a zero or a pole of order \( 2\sigma_j \in \mathbb{Z} \). The neutrality condition requires that the total sum of orders satisfies
    \[
    \sum_{j=1}^{n} 2 + \sum_{j=1}^{m} 2\sigma_j = -4,
    \]
    which ensures that \( Q(z)dz^2 \) is a global quadratic differential on \( \widehat{\mathbb{C}} \).
\end{enumerate}

Quadratic differentials \( Q(z)dz^2 \in \mathcal{QD}(\boldsymbol{x}, \boldsymbol{q}) \) must take the following form:
$$
Q(z)= 
\prod_{j=1}^{n}\left(z-x_j\right)^2\prod_{k=1}^{m}\left(z-q_k\right)^{2q_k},
$$
\end{defn}

In the main theorem (\ref{traces as horizontal trajectories}), we show that the traces of the multiple SLE(0) systems correspond precisely to the horizontal trajectories of this class of quadratic differentials $Q(z) \in \mathcal{QD}(\boldsymbol{x},\boldsymbol{q})$ with limiting ends at $\boldsymbol{x}=\{x_1,x_2,\ldots,x_n\}$.
\begin{thm}\label{traces as horizontal trajectories}
Let $\boldsymbol{x} = \{x_1, x_2, \ldots, x_n\} \subset \partial \mathbb{H}$ be a collection of distinct growth points on the real line, and let $\boldsymbol{q} = \{q_1, q_2, \ldots, q_m\} \subset \mathbb{C}$ be a conjugation-symmetric set of marked points in the complex plane. Consider the symmetric divisor
\[
\boldsymbol{\sigma} = \sum_{k=1}^{n} x_k + \sum_{j=1}^{m} \sigma_j \cdot q_j,
\]
where \( 2\sigma_j \in \mathbb{Z} \) for all \( j = 1, \ldots, m \), and the total charge satisfies the neutrality condition
\[
\sum_{j=1}^{n} 1 + \sum_{j=1}^{m} \sigma_j = -2.
\]

Then there exists a quadratic differential \( Q(z) \in \mathcal{QD}(\boldsymbol{z}) \), with zeros at $\boldsymbol{z}$ and poles of order $2\sigma_j$ at $q_j$, such that the Loewner hulls $K_t$ generated by the multiple Loewner flow with driving measure $\boldsymbol{\nu}(t)$ are contained in the horizontal trajectories of the quadratic differential \( Q(z)\,dz^2 \), with terminal directions tending toward the critical points $\boldsymbol{z}$. This description holds up to any time \( t \) strictly before the collision of any critical points or singularities.

Moreover, up to such a time \( t \), the pullback of the quadratic differential under the Loewner map satisfies
\[
Q(z) \circ g_t^{-1} \in \mathcal{QD}(\boldsymbol{z}(t), \boldsymbol{q}(t)),
\]
where $\boldsymbol{z}(t)$ denotes the evolving configuration of critical points under the flow at time \( t \).
\end{thm}

\begin{proof}
The proof is similar to the multiple chordal and multiple radial SLE(0) cases.
\end{proof}

The key ingredient in the proof of theorem (\ref{traces as horizontal trajectories}) is the integral of motion for the Loewner flows. This integral of motion, denoted by $N_t(z)$, arises as the classical limit of a martingale observable inspired by conformal field theory. For systematic and rigorous study of such conformal field theories, please refer to \cites{KM13,KM21}.
The key ingredient in the proof is to construct an integral of motion for the multiple chordal Loewner flows.
\begin{thm}\label{integral of motion in H}
Let $x_1,x_2,\ldots,x_n$ be distinct points and $u$ a marked point on the real line, for each $z \in \mathbb{C}$
\begin{equation}
N_t(z)=(g'_{t}(z))^2\prod_{k=1}^{n}(g_t(z)-x_k(t))^2\prod_{j=1}^{m}(g_t(z)-q_j(t))^{2\sigma_j}
\end{equation}
is an integral of motion on the interval $[0,\tau_t \wedge \tau)$ for the multiple chordal Loewner flow with parametrization $\nu_j(t)$.
\end{thm}
\begin{proof}
    By direction computation, similar to the multiple chordal SLE(0) and multiple radial SLE(0) cases.
\end{proof}

\newpage
\section{Examples} \label{examples}
In this section, we present several examples of multiple SLE(0) systems with multiple marked points. The main focus is to illustrate behaviors such as multiple trajectories asymptotically converging in the same direction, or exhibiting spiraling behavior without introducing spin.

We consider the multiple SLE(0) systems with symmetric charge distribution $\boldsymbol{q}= \sum_{j=1}^{n} x_j + \sum_{j=1}^{m} \sigma_j \cdot q_j$, where $\sigma_j \in \frac{1}{2}\mathbb{Z} $

For a multiple radial SLE(0) system with growth points $\boldsymbol{x} = \{x_1, \dots, x_n\}$ and marked points $\boldsymbol{q} = \{q_1, \dots, q_m\}$, the associated quadratic differential takes the form:
\[
Q(z)\,dz^2 = 
\prod_{j=1}^{n}(z - x_j)^2 \prod_{k=1}^{m}(z - q_k)^{2\sigma_k} \, dz^2,
\]
where $\sigma_k \in \tfrac{1}{2}\mathbb{Z}$ denotes the (half-integer) charge at $q_k$.

\begin{thm}
Let $Q(z) \in \mathcal{QD}(\boldsymbol{x},\boldsymbol{q})$ be a quadratic differential associated with the symmetric divisor $\boldsymbol{\sigma}$. Define a vector field $v_Q$ on $\widehat{\mathbb{C}}$ by
\[
v_Q(z) = \frac{1}{\sqrt{Q(z)}},
\]
where
\[
\sqrt{Q(z)} = 
\prod_{j=1}^{n}(z - x_j) \prod_{k=1}^{m}(z - q_k)^{\sigma_k}.
\]
The flow lines of the differential equation $\dot{z} = v_Q(z)$ coincide with the horizontal trajectories of the quadratic differential $Q(z)\,dz^2$.
\end{thm}

\begin{remark}
This result provides an elementary method to numerically visualize the horizontal trajectories of $Q(z)\,dz^2$, which correspond to the traces of the SLE(0) system.
\end{remark}

In the figures that follow, we map the upper half-plane \( \mathbb{H} \) to the unit disk \( \mathbb{D} \), and illustrate the traces in \( \mathbb{D} \). The zeros of \( Q(z) \) are marked in \textcolor{red}{red}, and the locations of the marked (charged) points are marked in \textcolor{green!60!black}{green}.

\subsection{Counterexample: two marked boundary points}

\begin{figure}[H] 
    \centering
    \includegraphics[width=15cm]{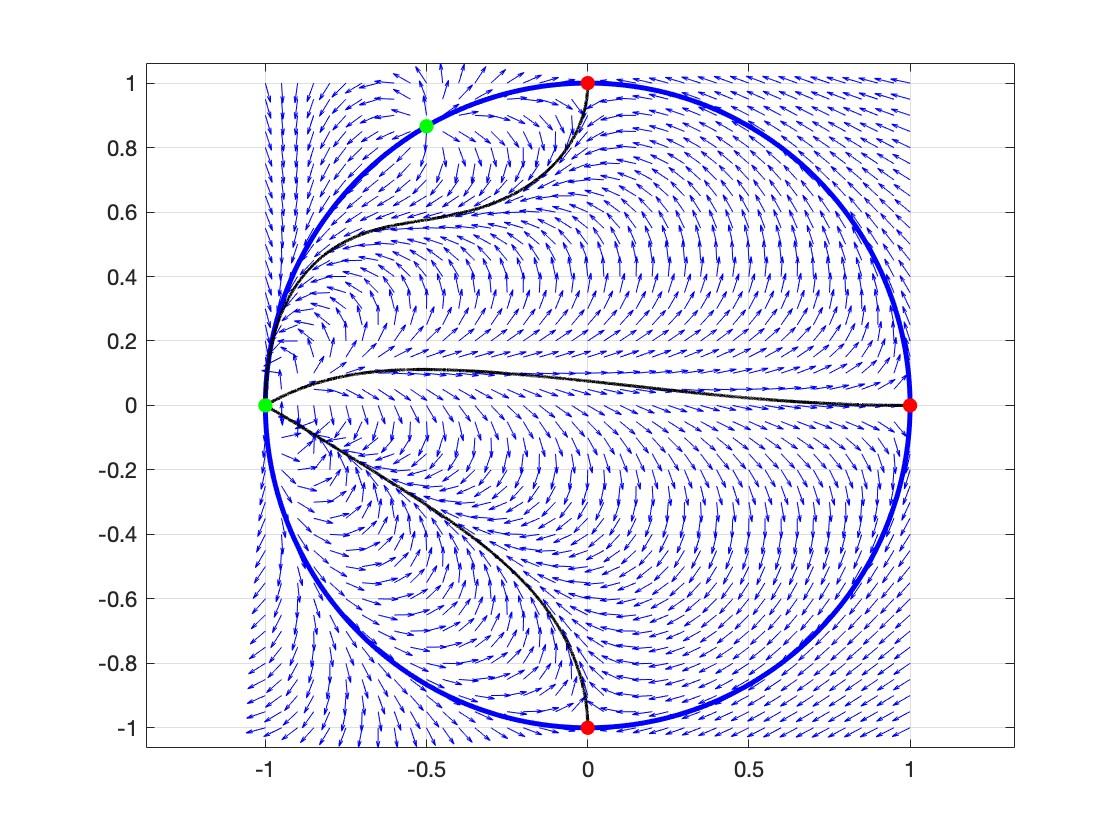}
    \caption{$x_1=-i$, $x_2 = 1$, $x_3 = i$, $q_1=\frac{2\pi i}{3}$, $q_2=-1$}
    \label{fig:two_boundary}
\end{figure}

In Figure~\ref{fig:two_boundary}, we consider the configuration with growth points \( x_1 = -i \), \( x_2 = 1 \), \( x_3 = i \), and marked points \( q_1 = \frac{2\pi i}{3} \), \( q_2 = -1 \). The symmetric divisor is given by
\[
\boldsymbol{\sigma} = x_1 + x_2 + x_3 - q_1 - 4 \cdot q_2.
\]
After normalization, the corresponding square root of the quadratic differential takes the form
\[
\sqrt{Q(z)} = C \cdot \frac{(z - i)(z + i)(z - 1)}{(z - e^{2\pi i / 3})(z + 1)^4},
\]
where \( C = 0.5003 - 0.8662i \) is a normalization constant.

The associated vector field is given by
\[
v_Q(z) = \frac{1}{\sqrt{Q(z)}},
\]
and the horizontal trajectories of the quadratic differential \( Q(z)\,dz^2 \) are the flow lines of this vector field.

In this example, the trajectory starting from \( x_3 = i \) asymptotically converges to the vertical tangent direction.

The MATLAB code used to generate this figure is included on the following page.
\begin{figure}[H]
    \centering
    \begin{minipage}{\linewidth}
        \centering
        \begin{lstlisting}[style=matlabIDE,language=Matlab, caption={MATLAB code for visualizing vector field in Figure~\ref{fig:two_boundary}}, label={code:two_boundary}]
% --- Draw unit circle ---
R = 1.0;
cx = 0; cy = 0;
n = 1000;
alpha = 0:pi/n:2*pi;
x = R*cos(alpha) + cx;
y = R*sin(alpha) + cy;
plot(x, y, '.b', 'LineWidth', 2);  % Plot boundary of the unit disk
grid on; axis equal; hold on;

% --- Coarse grid for quiver field ---
[u, v] = meshgrid(-1:0.05:1, -1:0.05:1);
z = u + v*i;

% Define the square root of the quadratic differential
f = (z - i).*(z + i).*(z - 1).*((z - exp(2*pi*i/3)).^(-1)).*((z + 1).^(-4));
g = exp(-3*pi*i/4)*f / (0.2582 + 0.9661*i);  % Normalized vector field

% Normalize 1/g to unit length
g1 = real(g) ./ sqrt(real(g).^2 + imag(g).^2);
g2 = -imag(g) ./ sqrt(real(g).^2 + imag(g).^2);

% Draw quiver arrows for the coarse field
quiver(u, v, g1, g2, 'color', 'b', 'MaxHeadSize', 0.05);


% --- Finer grid for streamline computation ---
[u, v] = meshgrid(-1:0.002:1, -1:0.002:1);
z = u + v*i;

% Redefine g(z) on finer grid
f = (z - i).*(z + i).*(z - 1).*((z - exp(2*pi*i/3)).^(-1)).*((z + 1).^(-4));
g = exp(-3*pi*i/4)*f / (0.0197 + 0.0737*i);  % Rescaled normalization
g1 = real(g)./sqrt(real(g).^2 + imag(g).^2);
g2 = -imag(g)./sqrt(real(g).^2 + imag(g).^2);

% --- Streamlines from starting points ---
startx = 0; starty = -0.98;
lineobj = streamline(u, v, g1, g2, startx, starty);  % From growth point at x_1 = -i
lineobj.LineWidth = 1.6; lineobj.Color = 'k';

startx = 0.98; starty = 0;
lineobj = streamline(u, v, -g1, -g2, startx, starty);  % From growth point at x_2 = 1
lineobj.LineWidth = 1.6; lineobj.Color = 'k';

startx = 0; starty = 0.98;
lineobj = streamline(u, v, g1, g2, startx, starty);  % From growth point at x_3 = i
lineobj.LineWidth = 1.6; lineobj.Color = 'k';

% --- Marked points  ---
plot(0, 1, '.r', 'MarkerSize', 20);                     % Growth point at x = i
plot(0, -1, '.r', 'MarkerSize', 20);                    % Growth point at x = -i
plot(1, 0, '.r', 'MarkerSize', 20);                     % Growth point at x = 1
plot(cos(2*pi/3), sin(2*pi/3), '.g', 'MarkerSize', 20); % Charge at e^{2pi*i/3}
plot(-1, 0, '.g', 'MarkerSize', 20);                    % Charge at x = -1
        \end{lstlisting}
    \end{minipage}
\end{figure}

\subsection{Counterexample: one marked boundary point, one marked interior point}

\begin{figure}[H]
    \centering
    \includegraphics[width=15cm]{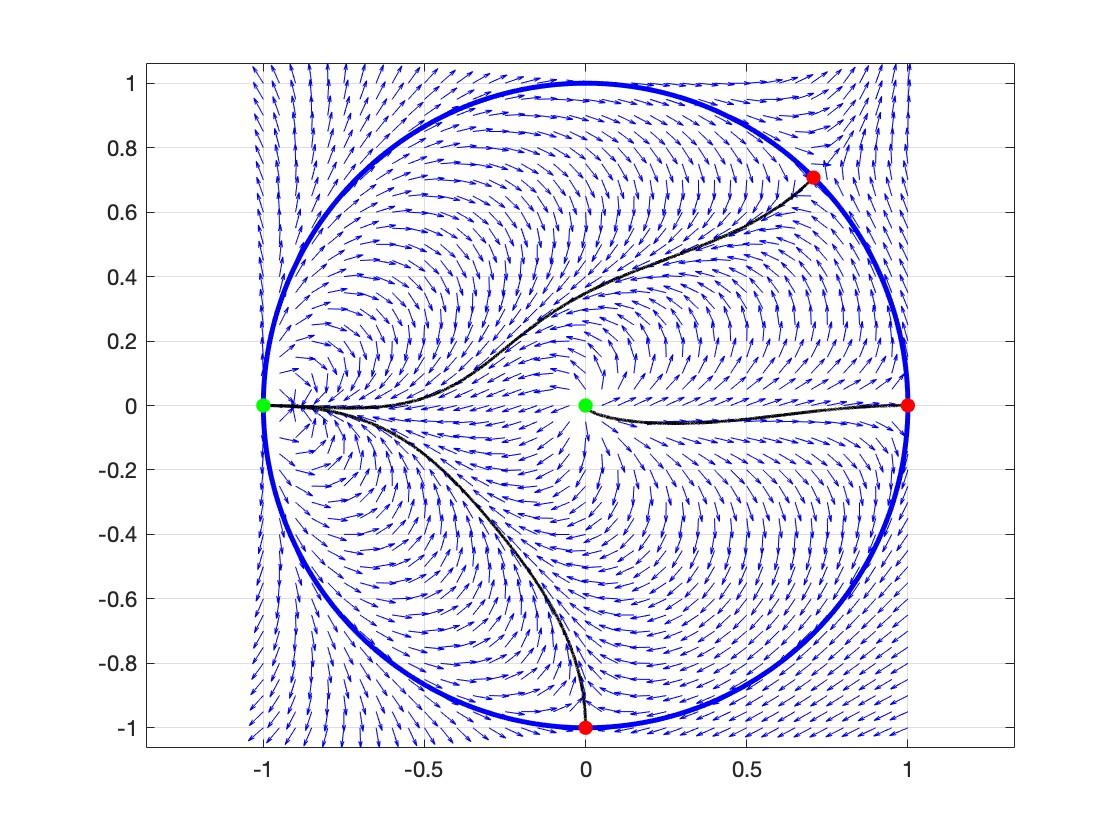}
    \caption{$x_1=-i$, $x_2=1$, $x_3=e^{\frac{\pi i}{4}}$, $q_1=1$, $q_2=-1$, $q_2^*=\infty$}
    \label{fig:boundary_interior}
\end{figure}

In Figure~\ref{fig:boundary_interior}, we consider the configuration with growth points \( x_1 = -i \), \( x_2 = 1 \), \( x_3 = e^{\frac{\pi i}{4}} \), and marked points \( q_1 = 0 \), \( q_2 = \infty \), \( q_3 = -1 \). The symmetric divisor is given by
\[
\boldsymbol{\sigma} = x_1 + x_2 + x_3 - q_1 - q_2 - 3\cdot q_3.
\]
After normalization, the square root of the associated quadratic differential is given by
\[
\sqrt{Q(z)} = C \cdot \frac{(z - e^{\pi i / 4})(z + i)(z - 1)}{z (z + 1)^3}.
\]
where \( C = -1.2071 - 0.5000i \) is the normalization constant.

The associated vector field is given by
\[
v_Q(z) = \frac{1}{\sqrt{Q(z)}},
\]
and the horizontal trajectories of the quadratic differential \( Q(z)\,dz^2 \) are the flow lines of \( v_Q(z) \).

In this example, two trajectories starting from $x_1 = -i$ and $x_3 = e^{\frac{\pi i}{4}}$ asymptotically converge to the same direction.
The figure is generated using \textsc{MATLAB}, and the corresponding code is provided on the following page.

\begin{figure}[H]
    \centering
    \begin{minipage}{\linewidth}
        \centering
        \begin{lstlisting}[style=matlabIDE, language=Matlab, caption={MATLAB code for visualizing vector field in Figure~\ref{fig:boundary_interior}}, label={code:boundary_interior}]
% --- Draw the unit circle ---
R = 1.0; cx = 0; cy = 0;
n = 1000;
alpha = 0:pi/n:2*pi;
x = R*cos(alpha) + cx;
y = R*sin(alpha) + cy;
plot(x, y, '.b', 'LineWidth', 2);   % Plot boundary of the unit disk
grid on; axis equal; hold on;

% --- Coarse vector field for quiver arrows ---
[u, v] = meshgrid(-1:0.05:1, -1:0.05:1);
z = u + v*i;

% Define the square root of the quadratic differential
f = (z - exp(pi*i/4)).*(z + i).*(z - 1).*(z.^(-1)).*((z + 1).^(-3));
g = f./(-0.7071 + 0.2929i);  % Normalized vector field

% Normalize 1/g to unit length
g1 = real(g) ./ sqrt(real(g).^2 + imag(g).^2);
g2 = -imag(g) ./ sqrt(real(g).^2 + imag(g).^2);

% Draw quiver arrows for the coarse field
quiver(u, v, g1, g2, 'color', 'b', 'MaxHeadSize', 0.05);


% --- Finer grid for streamline computation ---
[u, v] = meshgrid(-1:0.002:1, -1:0.002:1);
z = u + v*i;
f = (z - exp(pi*i/4)).*(z + i).*(z - 1).*(z.^(-1)).*((z + 1).^(-3));
g = f./(-0.7071 + 0.2929i);
g1 = real(g)./sqrt(real(g).^2 + imag(g).^2);
g2 = -imag(g)./sqrt(real(g).^2 + imag(g).^2);

% --- Streamlines from starting points ---
startx = 0.99*sqrt(2)/2; starty = 0.99*sqrt(2)/2;
lineobj = streamline(u, v, g1, g2, startx, starty);  % From growth point at x_1 = -i
lineobj.LineWidth = 1.6; lineobj.Color = 'k';

startx = 0; starty = -0.99;
lineobj = streamline(u, v, g1, g2, startx, starty);  % From growth point at x_1 = 1
lineobj.LineWidth = 1.6; lineobj.Color = 'k';

startx = 0.98; starty = 0;
lineobj = streamline(u, v, -g1, -g2, startx, starty);  % From growth point at x_3 = e^{pi*i/4}
lineobj.LineWidth = 1.6; lineobj.Color = 'k';

% --- Marked points  ---
plot(0, -1, '.r', 'MarkerSize', 20);                   % Growth point at x_1 = -i
plot(1, 0, '.r', 'MarkerSize', 20);                    % Growth point at x_2 = 1
plot(cos(pi/4), sin(pi/4), '.r', 'MarkerSize', 20);    % Growth point at x_3 = e^{pi*i/4}
plot(0, 0, '.g', 'MarkerSize', 20);                    % Charge at q = 0
plot(-1, 0, '.g', 'MarkerSize', 20);                   % Charge at q = -1
        \end{lstlisting}
    \end{minipage}
\end{figure}

\subsection{Counterexample: two marked interior points}

\begin{figure}[H]
    \centering
    \includegraphics[width=15cm]{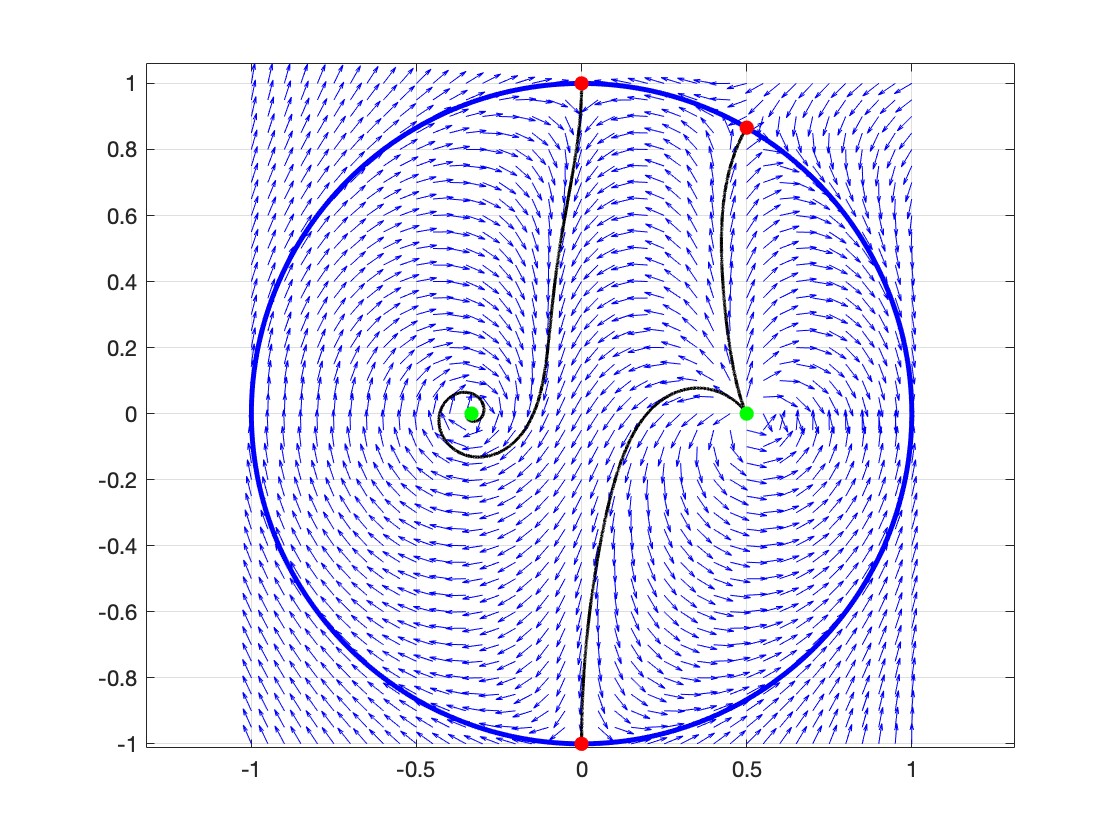}
    \caption{$x_1=-i$, $x_2= e^{\frac{\pi i}{3}}$, $x_3=i$, $q_1=\frac{1}{2}$, $q_1^*=2$, $q_2=-\frac{1}{3}$, $q_2^*=-3$}
    \label{fig:two_interior}
\end{figure}

In Figure~\ref{fig:two_interior}, we consider the configuration with growth points \( x_1 = -i \), \( x_2 = e^{\frac{\pi i}{3}} \), \( x_3 = i \), and marked points \( q_1 = -\frac{1}{3} \), \( q_2 = -3 \), \( q_3 = \frac{1}{2} \), \( q_4 = -2 \). The symmetric divisor is given by
\[
\boldsymbol{\sigma} = x_1 + x_2 + x_3 -  q_1 -  q_2 - \frac{3}{2}\cdot q_3 - \frac{3}{2}\cdot q_4 .
\]
After normalization, the square root of the associated quadratic differential is given by
\[
\sqrt{Q(z)} = C \cdot \frac{(z - e^{\pi i / 3})(z - i)(z + i)}{(z + \tfrac{1}{3})(z + 3)(z - 0.5)^{1.5}(z - 2)^{1.5}}.
\]
where \( C = i \cdot e^{-i\pi/6} \) is the normalization constant.

The associated vector field is given by
\[
v_Q(z) = \frac{1}{\sqrt{Q(z)}},
\]
and the horizontal trajectories of the quadratic differential \( Q(z)\,dz^2 \) are the flow lines of \( v_Q(z) \).

In this example, two trajectories—starting from \( x_1 = -i \) and \( x_2 = e^{\frac{\pi i}{3}} \)—asymptotically converge in the same direction. The trajectory starting from \( x_3 = i \) forms a spiral around the marked point \( q_1 = -\tfrac{1}{3} \), despite the absence of spin.

The figure is generated using \textsc{MATLAB}, and the corresponding code is provided on the following page.

\newpage

\begin{figure}[H]
    \begin{minipage}{\linewidth}
        \centering
        \begin{lstlisting}[language=Matlab, style=matlabIDE, caption={MATLAB code for visualizing the vector field in Figure~\ref{fig:two_interior}}, label={code:two_interior}]
% --- Draw unit circle ---
R = 1.0; cx = 0; cy = 0;
n = 1000;
alpha = 0:pi/n:2*pi;
x = R*cos(alpha) + cx;
y = R*sin(alpha) + cy;
plot(x, y, '.b', 'LineWidth', 2);  % Plot boundary of the unit disk
grid on; axis equal; hold on;

% --- Coarse vector field for quiver plot ---
[u, v] = meshgrid(-1:0.05:1, -1:0.05:1);
z = u + v*i;

% Define the square root of the quadratic differential
f = (z - exp(pi*i/3)).*(z - i).*(z + i) ...
    .* ((z + 1/3).^(-1)) .* ((z + 3).^(-1)) ...
    .* ((z - 0.5).^(-1.5)) .* ((z - 2).^(-1.5));
g = i*f / (exp(i*pi/6));  % Normalized vector field

% Normalize 1/g to unit length
g1 = real(g) ./ sqrt(real(g).^2 + imag(g).^2);
g2 = -imag(g) ./ sqrt(real(g).^2 + imag(g).^2);

% Draw quiver arrows for the coarse field
quiver(u, v, g1, g2, 'color', 'b', 'MaxHeadSize', 0.05);

% --- Finer grid for streamline computation ---
[u, v] = meshgrid(-1:0.002:1, -1:0.002:1);
z = u + v*i;

% Redefine g(z) on the finer grid
f = (z - exp(pi*i/3)).*(z - i).*(z + i) ...
    .* ((z + 1/3).^(-1)) .* ((z + 3).^(-1)) ...
    .* ((z - 0.5).^(-1.5)) .* ((z - 2).^(-1.5));
g = i*f / (exp(i*pi/6));
g1 = real(g) ./ sqrt(real(g).^2 + imag(g).^2);
g2 = -imag(g) ./ sqrt(real(g).^2 + imag(g).^2);

% --- Streamlines from starting points ---
startx = 0; starty = -0.98;
lineobj = streamline(u, v, -g1, -g2, startx, starty);  % From growth point at x_1 = -i
lineobj.LineWidth = 1.6; lineobj.Color = 'k';

startx = cos(pi/3)*0.98; starty = sin(pi/3)*0.98;
lineobj = streamline(u, v, -g1, -g2, startx, starty);  % From growth point at x_2 = e^{pi*i/3}
lineobj.LineWidth = 1.6; lineobj.Color = 'k';

startx = 0; starty = 0.98;
lineobj = streamline(u, v, g1, g2, startx, starty);  % From growth point at x_3 = i
lineobj.LineWidth = 1.6; lineobj.Color = 'k';

% --- Marked points  ---
plot(0, 1, '.r', 'MarkerSize', 20);                   % Growth point at x_1 = i
plot(0, -1, '.r', 'MarkerSize', 20);                  % Growth point at x_3 = -i
plot(cos(pi/3), sin(pi/3), '.r', 'MarkerSize', 20);   % Growth point at x_2 = e^{pi*i/3}
plot(1/2, 0, '.g', 'MarkerSize', 20);                 % Charge at q = 1/2
plot(-1/3, 0, '.g', 'MarkerSize', 20);                % Charge at q = -1/3
        \end{lstlisting}
    \end{minipage}
\end{figure}
\printbibliography
\end{document}